\documentclass[12pt]{article}%
\usepackage{graphicx}
\usepackage{amsmath}
\usepackage{amsfonts}
\usepackage{amssymb}%
\providecommand{\U}[1]{\protect\rule{.1in}{.1in}}
\newtheorem{theorem}{Theorem}

\newtheorem{claim}[theorem]{Claim}

\newtheorem{corollary}[theorem]{Corollary}

\newenvironment{proof}[1][Proof]{\noindent\textbf{#1.} }{\ \rule{0.5em}{0.5em}}
\begin{document}

\title{\textbf{Stationary States of the Generalized Jackson
Networks\footnotetext{\noindent Part of this work has been carried out in the
framework of the Labex Archimede (ANR-11-LABX-0033) and of the A*MIDEX project
(ANR-11-IDEX-0001-02), funded by the \textquotedblleft Investissements
d'Avenir" French Government programme managed by the French National Research
Agency (ANR). Part of this work has been carried out at IITP RAS. The support
of Russian Foundation for Sciences (project No. 14-50-00150) is gratefully
acknowledged.}}}
\author{Alexander Rybko$^{^{\ddag\ddag}}\,$ and Senya Shlosman$^{^{\ddag},^{\ddag
\ddag}}$\\$^{^{\ddag}}$Aix Marseille Universit\'{e}, Universit\'{e} de Toulon, \\CNRS, CPT UMR 7332, 13288, Marseille, France\\$^{^{\ddag\ddag}}$Inst. of the Information Transmission Problems,\\RAS, Moscow, Russia }
\maketitle

\begin{abstract}
We consider Jackson Networks on general countable graphs and with arbitrary
service times. We find natural sufficient conditions for existence and
uniqueness of stationary distributions. They generalise these obtained earlier
by Kelbert, Kontsevich and Rybko.

\textbf{Keywords:} Jackson Networks, countable Markov chains, uniqueness of
the stationary distribution, mean-field, double semi-stochastic matrices.

\end{abstract}

\section{Introduction}

This is a paper about open queuing networks. In our previous papers
\cite{RS1}, \cite{RS2} we were considering closed \textquotedblleft mean
field\textquotedblright\ queuing systems. The closedness of the system means
that the customers are never leaving the system, while new customers never
come to it. The \textquotedblleft mean field\textquotedblright\ condition
means that the network of $N$ servers forms a complete graph, and after being
served the customer is allowed to go for his next service to any of $N$
servers with uniform probability $\frac{1}{N}$. In addition, the random
service time $\eta$ is the same for all customers and for all servers.

In the present paper we will consider open networks, when customers are
leaving the system after several steps, while outside customers are coming for
service. We will show that under general conditions of not being overloaded
such systems always satisfy the Poisson Hypothesis (PH). We will relax the
mean-field symmetry of our system. (It is the largest symmetry possible,
corresponding to the action of the permutation group $S_{N}.$) Namely, we will
allow different values for the probabilities to go to different servers, as
well as different service times, $\eta_{i},$ depending on server.

The rough idea why PH always holds for our class of open systems is the
following. As we know from \cite{RS2}, the reason for the possible violation
of PH is that the memory about the initial state of the system is preserved,
to some degree. Since, however, every customer of the open system spends in it
only a finite average time, the memory of the initial state fades away as the
number of the customers initially present in the system goes to zero with time.

Here we will study one special class of networks, which are called Jackson
Networks. The Jackson Network -- JN -- is defined by the transition matrix
$P;$ the servers have exponential service times. The Generalized Jackson
Network -- GJN -- is defined by the transition matrix $P$ and arbitrary
service times. For the finite JN and GJN the conditions of existence and
uniqueness of the stationary distribution are known. For the infinite JN such
conditions were obtained by Kelbert, Kontsevich and Rybko, \cite{KKR}. For the
infinite GJN they were not known.

In this paper we find sufficient conditions of existence and uniqueness of the
stationary distribution for the mean-field limits of Generalized Jackson
Networks (both finite and infinite), i.e. for the corresponding Non-Linear
Markov Processes.

\section{Open systems}

\subsection{Finite number of server groups}

In the present section we consider the simplest case of the open system we can
treat. To define it we need to have a Markov chain $\mathfrak{M}_{m}$ with
$m+1$ states, the number $m\geq1$ being the number of different types of
servers in our network.

\subsubsection{Markov chain}

Let $P$ be a transition matrix of finite Markov chain $\mathfrak{M}_{m}$ with
$m+1$ states $1,2,...,m,\infty$, where the last state $\infty$ is absorbing.
We assume that all the matrix elements $p_{ij}>0$ for $1\leq i,j\leq m,$ and
that the probability $p_{i\infty}$ to get from the state $i$ to $\infty$ is
positive for at least \textbf{one} $i:$%
\begin{equation}
p_{i\infty}=1-\sum_{j=1}^{m}p_{ij}>0. \label{10}%
\end{equation}

\subsubsection{\label{network} The servers network}

Let $Nm$ be the total number of servers in our queuing network. We assume
further that to every group $i=1,...,m$ the random service time, $\eta_{i}$,
is assigned.

Once a customer starts her service at a server in the $i$-th group, it lasts a
random time $\eta_{i}.$ After that time the customer leaves the server. Then
with probability $p_{i\infty}$ the customer leaves the network, while she goes
to the type $j$ server with probability $p_{ij}.$ Within the group $j$ she
chooses one of its $N$ servers uniformly, with probability $\frac{1}{N}.$ If
it is occupied, the customer goes into the queue, being the last in it.

In addition, to every server of the $i$-th type there is assigned an inflow of
external customers, which is Poisson flow with the constant rate $v_{i}.$
Assuming that thus defined Markov process is ergodic, we denote by $\pi_{N}$
its stationary distribution, which is a measure on $\Omega^{Nm}$.

Clearly, in order to have the ergodicity of the network, we need some sort of
condition that the system is not overloaded. (In fact, it is sufficient for
the ergodicity, see \cite{FR}.) Let the vector $V=\left(  v_{1},...,v_{m}%
\right)  .$ We have to assume that the vector\
\begin{equation}
\bar{V}=V+VP+VP^{2}+..., \label{11}%
\end{equation}
which solves the equation on $\Lambda:$%
\begin{equation}
\Lambda=V+\Lambda P, \label{09}%
\end{equation}
satisfy for all $i$ the relation
\begin{equation}
\mathbb{E}\left(  \eta_{i}\right)  \bar{V}_{i}<1. \label{08}%
\end{equation}
(Note that under our hypothesis $\left(  \ref{10}\right)  $ the equation
$\left(  \ref{09}\right)  $ always has exactly one solution, which is given by
the (convergent) series $\left(  \ref{11}\right)  $.)

Since we are relying on the results of the paper \cite{RS1}, we have to impose
some conditions on the service time distributions. We will assume the
following properties: for each $i$

\begin{enumerate}
\item the density function $p_{i}\left(  t\right)  $ of $\eta_{i}$ is positive
on $t\geq0$ and uniformly bounded from above;

\item $p_{i}\left(  t\right)  $ satisfies the following strong Lipschitz
condition: for some $C<\infty$ and for all $t\geq0$
\[
\left\vert p_{i}\left(  t+\Delta t\right)  -p_{i}\left(  t\right)  \right\vert
\leq Cp_{i}\left(  t\right)  \left\vert \Delta t\right\vert ,
\]
provided $t+\Delta t>0$ and $\left\vert \Delta t\right\vert <1;$

\item introducing the random variables
\[
\eta_{i}\Bigm|_{\tau}=\left(  \eta_{i}-\tau\Bigm|\eta_{i}>\tau\right)
,\tau\geq0,
\]
we suppose that for some $\delta>0,$ $M_{\delta,\tau}<\infty$
\[
\mathbb{E}\left(  \eta_{i}\Bigm|_{\tau}\right)  ^{2+\delta}<M_{\delta,\tau}.
\]
Of course, this condition holds once
\[
M_{\delta}\equiv\mathbb{E}\left(  \eta_{i}\right)  ^{2+\delta}<\infty.
\]

\item the probability density $p_{i}\left(  t\right)  $ is differentiable in
$t,$ with $p_{i}^{\prime}\left(  t\right)  $ continuous. Moreover, let us
introduce the functions $p_{i,\tau}\left(  t\right)  ,$ which are the
densities of the random variables $\eta_{i}\Bigm|_{\tau},$ i.e.%

\[
p_{i,\tau}\left(  t\right)  =\frac{p_{i}\left(  t+\tau\right)  }{\int
_{0}^{\infty}p_{i}\left(  t+\tau\right)  dt}.
\]
We need that the function $p_{i,\tau}\left(  0\right)  $ is bounded uniformly
in $\tau\geq0$ (and in $i$ -- for the infinite network) while the function
$\frac{d}{d\tau}p_{i,\tau}\left(  0\right)  $ is continuous and bounded
uniformly in $\tau\geq0;$

\item the limits $\lim_{\tau\rightarrow\infty}p_{i,\tau}\left(  0\right)  ,$
$\lim_{\tau\rightarrow\infty}\frac{d}{d\tau}p_{i,\tau}\left(  0\right)  $
exist and are finite.\medskip
\end{enumerate}

In what follows we will always assume all these properties, \textbf{unless
stated otherwise}.

\subsubsection{ \label{wph} Weak PH}

In the limit as $N\rightarrow\infty$ we have a convergence to a system of
Non-Linear Markov Processes (NLMP), under proviso $\left(  \ref{08}\right)  $
above that the network is not overloaded. Informally, this process can be
described as follows.\textbf{ }It is the evolution of the collection of
measures $\mu_{1}\left(  t\right)  ,\mu_{2}\left(  t\right)  ,...,\mu
_{m}\left(  t\right)  ,$ which describe the states of the nodes $1,2,...,m.$
Each of the measures $\mu_{i}\left(  t\right)  $ is a probability distribution
on possible queues at the corresponding node $i$ at time $t.$ According to our
service rules each state $\mu_{i}\left(  t\right)  $ generates an exit
(non-Poissonian in general) flow from the node $i,$ having the rate function
$b_{i}\left(  t\right)  .$ On the other hand, the inflows to every node are
Poissonian with the rate functions $\lambda_{i}\left(  t\right)  =v_{i}%
+\sum_{j=1}^{m}b_{j}\left(  t\right)  p_{ji}.$ For more details the reader
should consult \cite{BRS} and \cite{RS1}.

The weak PH is the following statement:

\begin{theorem}
In the limit $N\rightarrow\infty$ the network described in Sect. \ref{network}
has the following properties:

1. the (total) flows of customers to different servers become independent;

2. the (total) flow of customers to any server of $i$-th type, $i=1,...,m,$
tends to a Poisson flow with the rate function $\lambda_{i}\left(  t\right)  $
(which depends on the initial state of our system);

3. the (\textbf{non}-Poissonian) limiting flow of customers from any server of
$i$-th type, $i=1,...,m,$ has rate function $b_{i}\left(  t\right)  $ (also
depending on the initial state), and for every $i$ we have
\begin{equation}
\lambda_{i}\left(  t\right)  =v_{i}+\sum_{j=1}^{m}b_{j}\left(  t\right)
p_{ji}. \label{006}%
\end{equation}
The statement holds provided the properties 4 and 5 of the service time
distributions are valid. We need no conditions on the initial state, and we do
not suppose the underload relation $\left(  \ref{08}\right)  $.
\end{theorem}

\noindent The above result is easily obtained by the technique of the paper
\cite{BRS} in the case of continuous distributions $\eta_{i}$. (In fact, our
situation is even simpler, since in \cite{BRS} we deal with the network where
the servers can exchange their positions.) The easier case of discrete random
variables $\eta_{i}$ can be treated along the lines of Appendix I of
\cite{RS2}.

\subsubsection{\label{sph} Strong PH}

The Strong Poisson Hypothesis, which is formulated below, is the statement
about the asymptotic independence of our network from its initial state. The
Strong Poisson Hypothesis means the validity of the following two theorems:

\begin{theorem}
\label{FSHP} Suppose the underload relation $\left(  \ref{08}\right)  $ holds.
There exist values $\hat{\lambda}_{i},\hat{b}_{i},$ depending only on the
rates $v_{i},$ service times $\eta_{i}$ and the matrix $P,$ such that for any
initial state $\varkappa=\otimes\varkappa_{i}$ of our network the limiting
behavior of the functions $\lambda_{i}\left(  t\right)  ,$ $b_{i}\left(
t\right)  $ as $t\rightarrow\infty,$ does not depend on the initial state
$\varkappa$ of the system, and moreover
\begin{equation}
\lambda_{i}\left(  t\right)  \rightarrow\hat{\lambda}_{i},\ b_{i}\left(
t\right)  \rightarrow\hat{b}_{i}\text{ as }t\rightarrow\infty. \label{004}%
\end{equation}

\end{theorem}

Let $\nu_{i}$ be the stationary distribution of the stationary Markov process
on a single server $\mathcal{N}_{i},$ corresponding to the stationary Poisson
input flow with constant rate $\hat{\lambda}_{i},$ and service time $\eta
_{i}.$ (Such a distribution exists provided the server $\mathcal{N}_{i}$ is
not overloaded.) Define $\hat{\pi}_{N}$ to be the product state on
$\Omega^{Nm}:$%
\[
\hat{\pi}_{N}=\prod_{i=1}^{m}\underset{N}{\underbrace{\nu_{i}\otimes
...\otimes\nu_{i}}}.
\]

\begin{theorem}
The set of limit points of the family $\pi_{N}$ contains at most one point,
which coincides with the limit $\lim_{N\rightarrow\infty}\hat{\pi}_{N}.$
\end{theorem}

Before giving the proof of the theorems we will explain why, in contrast with
the closed network case, we do not need to impose any condition on the initial
state of our network, when it is open. It is based on the following statement.

\begin{claim}
Let $\eta$ and $\xi_{i}$, $i=1,2,...$ be independent random variables, and
$\eta$ be an integer random variable. Then%
\[
\Xi=\sum_{i=1}^{\eta}\xi_{i}%
\]
also is a random variable. In particular, for every $\varepsilon>0$ there
exists a value $T\left(  \varepsilon\right)  <\infty$ such that
\[
\mathbf{\Pr}\left(  \Xi>T\left(  \varepsilon\right)  \right)  <\varepsilon.
\]

\end{claim}

\begin{proof}
Trivial.
\end{proof}

That claim explains that no matter what initial condition is chosen for our
network, after some time, depending on the condition, it is almost forgotten.
Now we will give the proof of the above theorems.

\begin{proof}
\medskip The flattening relation $\left(  \ref{004}\right)  $ for the rates
$\lambda_{i}\left(  t\right)  ,\ b_{i}\left(  t\right)  $ is proven in the
same way as in \cite{RS1}. First we recall the basic relation (26) of
\cite{RS1} between (arbitrary) input rate $\bar{\lambda}_{i}\left(  t\right)
$ and the corresponding output rate $\bar{b}_{i}\left(  t\right)  $ of the
server $\mathcal{N}_{i}:$
\begin{equation}
\bar{b}_{i}\left(  t\right)  =\left(  1-\varepsilon\left(  t\right)  \right)
\left[  \bar{\lambda}_{i}\ast q_{\bar{\lambda}_{i},t}\right]  \left(
t\right)  +\varepsilon\left(  t\right)  Q\left(  t\right)  . \label{005}%
\end{equation}
Here the functions $\varepsilon\left(  t\right)  ,$ $Q\left(  t\right)  $ and
the family $q_{\bar{\lambda}_{i},t}$ of probability densities are functionals
of the initial state of $\mathcal{N}_{i}$ and of the rate function
$\bar{\lambda}_{i};$ the function $Q\left(  t\right)  $ is uniformly bounded
in $t,$ the function $\varepsilon\left(  t\right)  $ goes to zero as
$t\rightarrow\infty,$ while each of the densities $q_{\bar{\lambda}_{i}%
,t}\left(  \cdot\right)  $ has its support on positive semi-axis and depends
on the function $\bar{\lambda}_{i}\left(  \tau\right)  $ only via its
restriction to $\left\{  \tau\leq t\right\}  .$ For more details the reader
can go to \cite{RS1}, Sect. 8.

To apply the relation $\left(  \ref{005}\right)  $ we need to have some
information on the regularity properties of the kernels $q_{\bar{\lambda}%
_{i},t}.$ In the situation of the closed systems, studied in \cite{RS1}, we
were using the uniform compactness of the family $\left\{  q_{\bar{\lambda
}_{i},t},t>0\right\}  :$ the integrals%
\begin{equation}
\int_{0}^{K}q_{\bar{\lambda}_{i},t}\left(  \tau\right)  ~d\tau\rightarrow1
\label{0023}%
\end{equation}
as $K\rightarrow\infty,$ \textit{uniformly in} $t.$ In the present situation
it will be sufficient for our purposes to establish the following weaker
property: there exists a function $K\left(  t\right)  ,$ such that%
\begin{equation}
\int_{0}^{K\left(  t\right)  }q_{\bar{\lambda}_{i},t}\left(  \tau\right)
~d\tau\rightarrow1,\text{with }t-K\left(  t\right)  \rightarrow\infty\text{ as
}t\rightarrow\infty. \label{0024}%
\end{equation}
We will show that the above property indeed holds, provided our network is
underloaded, see $\left(  \ref{07}\right)  $ or $\left(  \ref{08}\right)  $ below.

To get the relation $\left(  \ref{0024}\right)  $ we need to revert to the
definition of the densities $q_{\bar{\lambda}_{i},t},$ which is used in the
course of the proof of the Theorem 3 of \cite{RS1}. It is quite complicated,
but we need only some properties of it. According to this definition,
$q_{\bar{\lambda}_{i},t}$ is the density of a certain random variable
$\xi_{\bar{\lambda}_{i},t}\geq0,$ having the following property:

Let a realization $\left\{  x_{1},x_{2},...,x_{n}\right\}  \subset\left[
0,t\right]  $ of the Poisson random field with rate function $\bar{\lambda
}_{i}\left(  \tau\right)  ,$ $\tau\in\left[  0,t\right]  ,$ as well as the
(unordered) sequence $\left\{  l_{1},l_{2},...,l_{n+1}\right\}  $ are given,
where $l_{k}$ are iid random variables with distribution $\eta_{i}.$
\textit{Under this condition} the random variable $\xi_{\bar{\lambda}_{i},t}$
takes its values in the set
\[
L\left(  l_{1},l_{2},...,l_{n+1}\right)  =\left\{  \sum_{k\in A}l_{k}%
:A\subset\left\{  1,2,...,n\mathbf{+}1\right\}  ,A\neq\emptyset\right\}
\subset\mathbb{R}^{1},
\]
with probabilities, depending on the sample $\left\{  x_{1},x_{2}%
,...,x_{n}\right\}  .$ Therefore the (unconditioned) random variable
$\xi_{\bar{\lambda}_{i},t}$ is dominated from above by the random variable
\[
\sum_{k=1}^{\chi_{\bar{\lambda}_{i},t}+1}l_{k},
\]
where $l_{k}$ are iid random variables with distribution $\eta_{i},$ while
$\chi_{\bar{\lambda}_{i},t}$ is a Poisson random variable with parameter
$\int_{0}^{t}\bar{\lambda}_{i}\left(  \tau\right)  ~d\tau.$ So to get $\left(
\ref{0024}\right)  $ it is enough to show that under the conditions $\left(
\ref{07}\right)  $ or $\left(  \ref{08}\right)  $ below we have that for all
$t$
\begin{equation}
\mathbb{E}\left(  \chi_{\lambda_{i},t}\right)  \leq c\frac{t}{\mathbb{E}%
\left(  \eta_{i}\right)  }, \label{15}%
\end{equation}
with $c<1.$

Let us derive the \textquotedblleft flattening\textquotedblright\ relations
$\left(  \ref{004}\right)  ,$ assuming $\left(  \ref{0024}\right)  $ and
$\left(  \ref{005}\right)  .$ Applying $\left(  \ref{005}\right)  $ to rate
functions $\lambda_{j}\left(  t\right)  ,$ $b_{j}\left(  t\right)  ,$ and
using $\left(  \ref{0024}\right)  $ we obtain the relations%
\begin{equation}
b_{j}^{+}=\limsup_{t\rightarrow\infty}b_{j}\left(  t\right)  \leq\lambda
_{j}^{+}=\limsup_{t\rightarrow\infty}\lambda_{j}\left(  t\right)  , \label{31}%
\end{equation}%
\begin{equation}
b_{j}^{-}=\liminf_{t\rightarrow\infty}b_{j}\left(  t\right)  \geq\lambda
_{j}^{-}=\liminf_{t\rightarrow\infty}\lambda_{j}\left(  t\right)  . \label{32}%
\end{equation}
(The property $\left(  \ref{0024}\right)  $ is needed only for $\left(
\ref{32}\right)  .$) Indeed, the relation $\left(  \ref{005}\right)  $ is
telling us that the function $b\left(  t\right)  $ is the result of averaging
of $\lambda\left(  \cdot\right)  $ over a segment $\left[  t-K\left(
t\right)  ,t\right]  $ with some probabilistic kernel, while the left-end
point of the segment $t-K\left(  t\right)  \rightarrow\infty,$ as
$t\rightarrow\infty.$

Introducing now the vectors $\mathcal{L}=\left\{  \lambda_{j}^{+}-\lambda
_{j}^{-}\right\}  ,$ $\mathcal{B}=\left\{  b_{j}^{+}-b_{j}^{-}\right\}  ,$
$j=1,...,m,$ we have that $\mathcal{L}\geq\mathcal{B}$ coordinate-wice.
Applying $\limsup_{t\rightarrow\infty}$ and $\liminf_{t\rightarrow\infty}$ to
both sides of $\left(  \ref{006}\right)  $ we get
\[
\mathcal{L}\leq\mathcal{B}P.
\]
Therefore
\[
\mathcal{L}\leq\mathcal{L}P,
\]
and so for all $n>0$
\[
\mathcal{L}\leq\mathcal{L}P^{n}.
\]
Due to condition $\left(  \ref{10}\right)  ,$ for every $x=\left(
x_{1},...,x_{m}\right)  $ with $x_{i}\geq0$
\begin{equation}
\left\vert \left\vert xP^{n}\right\vert \right\vert _{L^{1}}\leq c\left\vert
\left\vert x\right\vert \right\vert _{L^{1}} \label{12}%
\end{equation}
for all $n\geq m$ with some $c<1.$ Thus, we conclude that $\mathcal{L}=0.$
That proves $\left(  \ref{004}\right)  .$

To have the existence of the stationary measures and their convergence, as
well as the regularity $\left(  \ref{0024}\right)  ,$ we need the system to be
underloaded. In case when we have just one type of servers, i.e. $i=1,$ this
condition is simply the requirement that%
\begin{equation}
\mathbb{E}\left(  \eta_{1}\right)  \frac{v_{1}}{1-p}<1. \label{07}%
\end{equation}
Here $p=p_{11}<1$ is the probability that the client will return back to the
server after being served. The ratio $\frac{v_{1}}{1-p}\equiv v_{1}%
+v_{1}p+v_{1}p^{2}+...$ has the following meaning: it is the mean number of
the customers who will visit a given server $\mathcal{N}_{1}$ and who are the
descendants of clients who first entered the network during the unit time
interval $\left[  0,1\right]  .$ (The flow of these customers is not a Poisson
flow.) Note that some of these customers visit $\mathcal{N}_{1}$ only very
late in time. Therefore the condition $\left(  \ref{07}\right)  $ implies that
the average number of clients per given server times the average service time
is less than $1,$ since it is bounded from above by $\mathbb{E}\left(
\eta_{1}\right)  \frac{v_{1}}{1-p}.$ Therefore $\left(  \ref{0024}\right)  $ holds.

The existence of the stationary measures under $\left(  \ref{08}\right)  $ is
a result of the paper \cite{FR}. Let us derive the relation $\left(
\ref{15}\right)  .$ It is almost immediate after what was said in the two
preceding paragraphs. Indeed, the expectation $\mathbb{E}\left(  \chi
_{\lambda_{i},t}\right)  $ is nothing else as the mean value of the customers
visiting a given server $\mathcal{N}_{i}$ during the time interval $\left[
0,t\right]  .$ (Moreover, now they even form a Poisson flow!) As was explained
above, the estimate $\mathbb{E}\left(  \chi_{\lambda_{i},t}\right)  \leq
\bar{V}_{i}t$ holds. (In fact, the inequality is strict, since some customers
will come to the service much later than $t.$) Together with $\left(
\ref{08}\right)  $ it implies $\left(  \ref{15}\right)  .$
\end{proof}

\subsection{Infinite number of server groups}

In this subsection we explain the changes needed in order to extend the
results of the previous Section to the case of countably many servers. As
above, it will be built on the Markov chain $\mathfrak{M},$ which now will be countable.

\subsubsection{Markov chain}

We will need some condition of transience of $\mathfrak{M},$ analogous to
$\left(  \ref{10}\right)  .$ The condition $\left(  \ref{10}\right)  $ was
used to derive the relation $\left(  \ref{12}\right)  ,$ which means that the
only invariant measure of our chain is zero measure. The proper analog of it
in the present context follows. It turns out to be the condition of vanishing
of $L^{\infty}$-invariant measures (they do not need to be probability
measures). The corresponding sufficient conditions on the transition matrix
are given by Theorems \ref{T01}, \ref{T04} or \ref{T05} below.

Let $\mathfrak{M}$ be a countable irreducible Markov chain with the states
$i=1,2,...\cup\infty.$ Let $P=\left\{  p_{ij}\geq0\right\}  $ be the
transition matrix, $\sum_{j}p_{ij}\leq1,$ while
\[
p_{i\infty}=1-\sum_{j}p_{ij}\geq0
\]
be the probabilities to go from $i$ to $\infty.$ The state $\infty$ is absorbing.

The measure $\lambda=\left\{  \lambda\left(  i\right)  \geq0\right\}  $ is
called $L^{\infty}$-measure, if $\lambda\left(  i\right)  \leq C,$ for some
$C>0$, uniformly in $i.$ The measure $\lambda$ is called invariant for the
Markov chain $\mathfrak{M}$, if
\[
\lambda=\lambda P.
\]
i.e. if $\lambda\left(  j\right)  =\sum_{i}\lambda\left(  i\right)  p_{ij}.$
For example, the measure $\lambda=0$ is invariant.

The first such transience condition was obtained in \cite{KKR}, and it applies
only to double semi-stochastic matrices $P.$ We recall that $P$ is called
double semi-stochastic, if for all $j$ we have
\[
\sum_{i}p_{ij}\leq1.
\]
In this case there exists one special -- `maximal' -- invariant measure
$\lambda^{\ast}$ of $\mathfrak{M.}$ It is constructed as follows. Consider the
dual Markov chain $\mathfrak{M}^{\ast}$ on $1,2,...\cup~\infty,$ with
transition probabilities $p_{ij}^{\ast}=p_{ji},$ and with
\[
p_{i\infty}^{\ast}=1-\sum_{j}p_{ij}^{\ast}\geq0,\ \ p_{\infty\infty}^{\ast
}=1,
\]
i.e. $\infty$ is the absorbing state. Define the measure $\lambda^{\ast}$ by
\begin{equation}
\lambda^{\ast}\left(  i\right)  =\mathbf{\Pr}\left\{  \text{the chain
}\mathfrak{M}^{\ast},\text{ started from }i,\text{ never gets to }%
\infty\right\}  . \label{42}%
\end{equation}
Clearly, $\lambda^{\ast}\left(  i\right)  =\sum_{j}p_{ij}^{\ast}\lambda^{\ast
}\left(  j\right)  \equiv\sum_{j}\lambda^{\ast}\left(  j\right)  p_{ji},$ so
$\lambda^{\ast}$ is an invariant measure for $\mathfrak{M}$.

The following theorem gives sufficient condition for the zero measure to be
the only invariant measure.

\begin{theorem}
\label{T01} (see \cite{KKR}.) Suppose that the transition matrix $P$ of the
irreducible Markov chain $\mathfrak{M}$ is double semi-stochastic. Then the
following two properties are equivalent:

1. the invariant measure $\lambda^{\ast}$ $\left(  \ref{42}\right)  $ of the
Markov chain $\mathfrak{M}$ is zero.

2. zero measure is the only invariant $L^{\infty}$-measure of $\mathfrak{M}$.
\end{theorem}

The following proof is simpler than the original one, see \cite{KKR}, and
easily leads to generalizations, which follow.

\begin{proof}
Let us write a formula for the measure $\lambda^{\ast}.$ To this end denote by
$\gamma_{j}^{0}\left(  i\right)  $ the function%
\begin{equation}
\gamma_{j}^{\left(  0\right)  }\left(  i\right)  =\left\{
\begin{array}
[c]{ll}%
1 & \text{ if }i=j\\
0 & \text{ otherwise}%
\end{array}
\right.  , \label{01}%
\end{equation}
and put
\begin{equation}
\gamma_{j}^{\left(  n\right)  }=P\gamma_{j}^{\left(  n-1\right)  }. \label{02}%
\end{equation}
$\gamma$-s are column-vectors, $\gamma_{j}^{\left(  1\right)  }$ being the
$j$-th column of $P.$ The value $\gamma_{j}^{\left(  n\right)  }\left(
i\right)  $ is the probability for $\mathfrak{M}^{\ast},$ started at $j,$ to
be in $i$ after $n$ steps. Therefore the sum $\sum_{i}\gamma_{j}^{\left(
n\right)  }\left(  i\right)  $ is the probability that $\mathfrak{M}^{\ast},$
started at $j,$ is not at $\infty$ after $n$ steps. Evidently,%
\[
\sum_{i}\gamma_{j}^{\left(  n\right)  }\left(  i\right)  =eP^{n}\gamma_{j}%
^{0},
\]
where (the measure) $e$ is given by $e\left(  i\right)  \equiv1.$ The
probabilities $\lambda^{\ast}\left(  j\right)  $ are just the limits
\[
\lambda^{\ast}\left(  j\right)  =\lim_{n\rightarrow\infty}eP^{n}\gamma_{j}%
^{0}.
\]
They exist because for every $j$ the sequence $eP^{n}\gamma_{j}^{0},$
$n=0,1,2,...$ is non-increasing.

Suppose now that $\lambda^{\ast}=0.$ That means that $\lim_{n\rightarrow
\infty}\left(  eP^{n}\right)  \left(  j\right)  =0$ for every $j.$ If $h\geq0$
is an invariant $L^{\infty}$ measure, $h=hP,$ then for some $C$ we have $h\leq
Ce.$ But then evidently $h\left(  j\right)  \leq C\left(  eP^{n}\right)
\left(  j\right)  ,$ so $h$ has to be zero.
\end{proof}

\begin{corollary}
\label{C} Let $P$ satisfies the conditions of the previous theorem, and
$k\geq0$ be any $L^{\infty}\ $measure. Then $kP^{n}\rightarrow0$ weakly, i.e.
$\left(  kP^{n}\right)  \left(  j\right)  \rightarrow0$ for every $j$ (though
not uniformly in $j$).
\end{corollary}

\begin{proof}
The proof is contained in the proof of the theorem \ref{T01}.
\end{proof}

Note that in fact we have proven a stronger result, which generalize the above
theorem to the class of stochastic matrices $P$, which, instead of being
double-stochastic, have their columns summable.

\begin{theorem}
\label{T04} Suppose that $P$ is a stochastic matrix, such that for every $j$
the function $\gamma_{j}^{\left(  n\right)  },$ defined by relations $\left(
\ref{01}\right)  ,$ $\left(  \ref{02}\right)  $ belongs to $L^{1}$ once $n\geq
n_{0}.$ Suppose moreover that the limits%
\[
\lambda^{\ast}\left(  j\right)  =\lim_{n\rightarrow\infty}\left\vert
\left\vert \gamma_{j}^{\left(  n\right)  }\right\vert \right\vert _{L^{1}}%
\]
exist.

If $\lambda^{\ast}\left(  j\right)  =0$ for all $j,$ then for every
$L^{\infty}$ measure $k\geq0$ we have $kP^{n}\rightarrow0$ pointwise. So in
particular zero is the only invariant measure.
\end{theorem}

Still stronger statement holds as well.

\begin{theorem}
\label{T05} Suppose that $P$ is a matrix with non-negative entries,
$p_{ij}\geq0$, such that for all $j$ the functions $\gamma_{j}^{\left(
n\right)  },$ defined by relations $\left(  \ref{01}\right)  ,$ $\left(
\ref{02}\right)  ,$ belong to $L^{1}$ once $n\geq n_{0}.$ Suppose that the
convergence
\[
\left\vert \left\vert \gamma_{j_{0}}^{\left(  n\right)  }\right\vert
\right\vert _{L^{1}}\rightarrow0\text{ as }n\rightarrow\infty
\]
holds for just one value $j=j_{0},$ and suppose also that all the
corresponding matrix elements $p_{ij_{0}}$ are positive. Then we have the
convergence%
\[
\left\vert \left\vert \gamma_{j}^{\left(  n\right)  }\right\vert \right\vert
_{L^{1}}\rightarrow0\text{ as }n\rightarrow\infty
\]
for all other values of $j,$ so in particular all the conclusions of the
preceding Theorem holds.
\end{theorem}

\begin{proof}
In the notation of the proof of the Theorem \ref{T01}, the vectors $eP^{n}$
are well defined once $n\geq n_{0}.$ Indeed, $\left(  eP^{n}\right)  \left(
j\right)  \equiv\left\vert \left\vert \gamma_{j}^{\left(  n\right)
}\right\vert \right\vert _{L^{1}}.$ By our assumption we have $eP^{n}%
\gamma_{j_{0}}^{\left(  0\right)  }\rightarrow0$ as $n\rightarrow\infty.$ But
that implies immediately that also $eP^{n}\gamma_{j_{0}}^{\left(  1\right)
}\rightarrow0$ as $n\rightarrow\infty$. Evidently,
\[
eP^{n}\gamma_{j_{0}}^{\left(  1\right)  }=\sum_{i}p_{ij_{0}}\ eP^{n}\gamma
_{i}^{\left(  0\right)  }.
\]
Since $p_{ij_{0}}>0$ for all $i,$ the convergence $\sum_{i}p_{ij_{0}}%
\ eP^{n}\gamma_{i}^{\left(  0\right)  }\rightarrow0$ as $n\rightarrow\infty$
implies that $eP^{n}\gamma_{i}^{\left(  0\right)  }\rightarrow0$ for every
$i.$
\end{proof}

The claim of the Corollary \ref{C} is still valid.

\subsubsection{\medskip The servers network}

We suppose that the network consists of $Nm$ servers of $m=m\left(  N\right)
$ types, with $m=m\left(  N\right)  \rightarrow\infty$ as $N\rightarrow
\infty.$ Again, to every type $i$ the random service time, $\eta_{i}$, is
assigned. The probability of going from server of type $i$ to type $j$ is
$p_{ij},$ while the probability of leaving the system is $\tilde{p}_{i\infty
}=1-\sum_{j}p_{ij}.$ Within the type $j$ the customer chooses the server
uniformly. The notations $v_{i},$ $\pi_{N}$ have the same meaning as above.

\subsubsection{ Weak PH}

In the limit as $N\rightarrow\infty$ we again have a convergence to a system
of Non-Linear Markov Processes (NLMP), see \cite{BRS}. All the claims of the
Section $\left(  \ref{wph}\right)  $ remains true. The only difference is that
now we have infinitely many NLMP-s.

\subsubsection{Strong PH}

Strong PH holds here as well, in the sense of Section $\left(  \ref{sph}%
\right)  .$ The proof of the claim $\left(  \ref{004}\right)  $ proceeds in
the same way as there. The contraction property $\left(  \ref{12}\right)  $ of
the operator $P$ is now ensured by the Theorems \ref{T04}, \ref{T05}, see
Corollary \ref{C}. In order to use the contraction property we need to know in
advance that the network is underloaded. That is, we need to know that the
initial state is forgotten after some time, and that the inflow rates are not
too high. The conditions on inflow rates are very similar to those of the
Theorem \ref{FSHP}. Namely, let $\mathbf{V=}\left\{  v_{i}\right\}  $ be the
vector of the (constant) inflow rates; we need that the vector
\begin{equation}
\mathbf{\bar{V}}=\mathbf{V}+\mathbf{V}P+\mathbf{V}P^{2}+... \label{54}%
\end{equation}
satisfies for each $i$ the inequality%
\begin{equation}
\mathbb{E}\left(  \eta_{i}\right)  \mathbf{\bar{V}}_{i}<1. \label{53}%
\end{equation}
Let $\varkappa_{i}$ be the initial states of our network. The only condition
we need is that at every node we have a \textit{finite} random queue, which
means (tautologically) that the probability of infinite queue is $0.$ The
reason is that for the networks defined by the matrix $P$ satisfying the
conditions of Theorem \ref{T01}, Theorem \ref{T04} or Theorem \ref{T05} the
network becomes underloaded after some finite time, provided $\left(
\ref{53}\right)  $ holds. It is proven in \cite{KKR} for networks satisfying
the conditions of Theorem \ref{T01} and having exponential service times. In
the general case the proof is the same. So the initial state does not play any
role in the asymptotic state of our network, as it was the case for the finite networks.

The convergence property $\left(  \ref{54}\right)  $ hold, for example, for
any $\mathbf{V}$ from $L^{1}$ and any $P$ satisfying Theorem \ref{T01}, since
it can be shown that in this situation the matrix $P^{\ast}$ is evidently transient.

\textbf{ }Another example is if $\mathbf{V}$ is from $L^{\infty}$ and the set
$\mathrm{supp}\left(  \mathbf{V}\right)  $ is non-massive for the chain
$\mathfrak{M}^{\ast},$ defined for $P$ satisfying Theorem \ref{T01}. We recall
that a set $A$ is called non-massive for the chain $\mathfrak{M}^{\ast},$ if
the probability that the chain never hits $A$ is positive. For the proof see
\cite{DY}, Sect. 1.8.

We conclude by stating the main theorem of this section.

\begin{theorem}
\label{T02} Consider the network, defined by the matrix $P,$ satisfying the
conditions of either Theorem \ref{T01}, Theorem \ref{T04} or Theorem
\ref{T05}. Suppose the input rates satisfy the relations $\left(
\ref{54}\right)  ,$ $\left(  \ref{53}\right)  .$

There exist values $\hat{\lambda}_{i},\hat{b}_{i},$ depending only on the
rates $v_{i},$ service times $\eta_{i}$ and the matrix $P,$ such that for any
initial state $\varkappa=\otimes\varkappa_{i}$ of our network with finite
queues the limiting behavior of the functions $\lambda_{i}\left(  t\right)  ,$
$b_{i}\left(  t\right)  $ as $t\rightarrow\infty,$ does not depend on the
initial state $\varkappa$ of the system, and moreover for every $i$%
\[
\lambda_{i}\left(  t\right)  \rightarrow\hat{\lambda}_{i},\ b_{i}\left(
t\right)  \rightarrow\hat{b}_{i}\text{ as }t\rightarrow\infty.
\]

\end{theorem}

\begin{proof}
After all the remarks made above, the proof goes in the same way as for the
Theorem \ref{FSHP}, and is therefore omitted.
\end{proof}

\section{Closed systems}

Here we again consider the same situation as in the Section 2, but now the
chain $\mathfrak{M}_{m}$ has $m$ states, with transition matrix $P,$ and all
the exit probabilities $p_{i\infty}$ are zero. That is, our network is closed;
all exterior flow rates $v_{i}$ are then equal to zero. We assume additionally
that all matrix elements $p_{ij}>0,$ so the chain is ergodic. Again we will
study the mean-field type model, where we interconnect $N$ copies of our
network. For every $N$ we will have fixed number of customers, $K,$ and we
will consider the limit when $N,K\rightarrow\infty,$ so that $\frac{K}%
{Nm}\rightarrow\rho.$

For every $N$ we have some initial conditions, and we suppose that as
$N\rightarrow\infty,$ they converge to the limit, which will be the initial
condition $\varkappa=\otimes\varkappa_{i}$ for the Non-Linear Markov Process.
Now, in contrast with the open systems, we need to impose some restrictions on
$\varkappa$ in order to have strong PH. It is the same condition which
appeared already in \cite{RS1} -- the finiteness of the expected service times
$S\left(  \varkappa_{i}\right)  .$ They are defined as follows.

Consider the function
\[
R_{\eta_{i}}\left(  \tau\right)  =\mathbb{E}\left(  \eta_{i}\Bigm|_{\tau
}\right)  ,
\]
which is the expected (remaining) time of the service of the client, who
already spent the time $\tau$ in the server. For a queue $\omega=\left(
n,\tau\right)  ,$ containing $n$ clients, one of whom is already served for
the time $\tau,$ we define its expected service time, $S\left(  \omega\right)
,$ by%
\[
S\left(  \omega\right)  =\left\{
\begin{array}
[c]{ll}%
0 & \text{ for }\omega=\mathbf{0,}\\
\left(  n-1\right)  \mathbb{E}\left(  \eta_{i}\right)  +R_{\eta_{i}}\left(
\tau\right)  & \text{for }\omega=\left(  n,\tau\right)  ,\text{with }n>0.
\end{array}
\right.
\]
We then define the expected service time $S\left(  \varkappa\right)  $ as
$\mathbb{E}_{\varkappa}\left(  S\left(  \omega\right)  \right)  .$

Note that if the total expected service time $\sum_{i}S\left(  \varkappa
_{i}\right)  $ is finite, then so is the total expected number of clients in
our network.

In the following we will sketch the proof of the Strong PH for the above
setting. The main ideas are contained in \cite{RS1}.

First of all, we have the balance relation:%
\begin{equation}
\lambda_{i}\left(  t\right)  =\sum_{j}b_{j}\left(  t\right)  p_{ji}.
\label{69}%
\end{equation}
We also have for all $i=1,2,...,m:$
\begin{equation}
b_{i}\left(  t\right)  =\left(  \lambda_{i}\ast q_{i,\lambda_{i},t}\right)
\left(  t\right)  +\varepsilon_{i}\left(  t\right)  , \label{61}%
\end{equation}
where $\varepsilon_{i}\left(  t\right)  \rightarrow0$ as $t\rightarrow\infty.$
(This is the analog of the relation $\left(  \ref{005}\right)  $). The case
$m=1$ is the one treated in \cite{RS1}. The relation $\left(  \ref{61}\right)
$ boils then down to
\begin{equation}
\lambda\left(  t\right)  =\left(  \lambda\left(  \cdot\right)  \ast
q_{\lambda\left(  \cdot\right)  ,t}\right)  \left(  t\right)  +\varepsilon
\left(  t\right)  . \label{68}%
\end{equation}

In \cite{RS1} we were able to show that some apriori properties of the
function $\lambda$ imply corresponding properties of the stochastic kernels
$q_{\lambda,t},$ which in turn imply the convergence of the function
$\lambda\left(  t\right)  $ to the limit value as $t\rightarrow\infty$. The
properties of the kernels needed are the following:

\begin{enumerate}
\item for every $\varepsilon>0$ there exists a value $K\left(  \varepsilon
\right)  ,$ such that
\begin{equation}
\int_{0}^{K\left(  \varepsilon\right)  }q_{\lambda\left(  \cdot\right)
,t}\left(  x\right)  \,dx\geq1-\varepsilon\label{111}%
\end{equation}
uniformly in $t,\lambda\left(  \cdot\right)  .$

\item For every $T$ the (monotone continuous) function
\begin{equation}
F_{T}\left(  \delta\right)  =\inf_{x\geq X\left(  T\right)  }\inf
_{\substack{D\subset\left[  0,T\right]  :\\\mathrm{mes}D\geq\delta}}\int
_{D}q_{\lambda\left(  \cdot\right)  ,t}\left(  x\right)  \,dx \label{112}%
\end{equation}
is positive once $\delta>0,$ for some choice of the function $X\left(
T\right)  <\infty$ -- compare with relations (99), (100) of \cite{RS1}.
\end{enumerate}

\noindent They were derived in \cite{RS1} from the fact that in this situation
the probability that the node is empty becomes positive after some (long)
time. This property, in turn, follows from the fact that the mean number of
clients is conserved in the NLMP.

For $m>1$ the situation is a bit more complex. To treat it, let us introduce
the families $\Gamma_{i,n}$ of the trajectories of the chain $\mathfrak{M}%
_{m},$ $i=1,2,...m,$ $n\geq1.$ The family $\Gamma_{i,n}$ consists of all loops
$\gamma\left(  t\right)  \in\left\{  1,...,m\right\}  ,$ $t=0,1,...,n,$ such
that $\gamma\left(  0\right)  =\gamma\left(  n\right)  =i,$ while
$\gamma\left(  k\right)  \neq i$ for $k=1,...,n-1.$ For $\gamma\in\Gamma
_{i,n}$ let us define
\[
p\left(  \gamma\right)  =\prod_{k=1}^{n}p_{\gamma\left(  k-1\right)
\gamma\left(  k\right)  }.
\]
Then, for all $i$ we have
\[
\sum_{n}\sum_{\gamma\in\Gamma_{i,n}}p\left(  \gamma\right)  =1,
\]
since the last sum is precisely the probability of the event to return to the
state $i,$ starting from it. It follows that for all $i$%
\begin{align}
\lambda_{i}\left(  t\right)   &  =\sum_{n}\sum_{\gamma\in\Gamma_{i,n}}p\left(
\gamma\right)  \underset{n}{\underbrace{\int...\int}}q_{\gamma\left(
n-1\right)  ,\lambda_{\gamma\left(  n-1\right)  },t}\left(  x_{1}\right)
\label{62}\\
&  q_{\gamma\left(  n-2\right)  ,\lambda_{\gamma\left(  n-2\right)  },t-x_{1}%
}\left(  x_{2}\right)  ...q_{\gamma\left(  0\right)  ,\lambda_{\gamma\left(
0\right)  },t-x_{1}-...-x_{n-1}}\left(  x_{n}\right) \nonumber\\
&  \lambda_{i}\left(  t-x_{1}-...-x_{n}\right)  dx_{1}...dx_{n}.\nonumber
\end{align}
Indeed, from $\left(  \ref{69}\right)  $ and $\left(  \ref{61}\right)  $ we
have, ignoring the $\varepsilon$\--term, that
\begin{equation}
\lambda_{i}\left(  t\right)  =\sum_{j}\left(  \lambda_{j}\ast q_{i,\lambda
_{j},t}\right)  \left(  t\right)  p_{ji}. \label{71}%
\end{equation}
Let us fix an index $i,$ write this expression for all $j\neq i$ and insert
the resulting representations in $\left(  \ref{71}\right)  .$ Iterating this
procedure we get $\left(  \ref{62}\right)  .$ The advantage of $\left(
\ref{62}\right)  $ over $\left(  \ref{61}\right)  $ is that it expresses each
function $\lambda_{i}\left(  t\right)  $ via itself at earlier times, as is
the case in relation $\left(  \ref{68}\right)  ,$ via convolution with the
stochastic kernels%
\[
Q_{\left\{  \lambda_{1},...,\lambda_{m}\right\}  ,t}\left(  X\right)
=\sum_{\gamma}p\left(  \gamma\right)  Q_{\left\{  \lambda_{1},...,\lambda
_{m}\right\}  ,\gamma,t}\left(  X\right)  ,
\]
where
\begin{align*}
&  Q_{\left\{  \lambda_{1},...,\lambda_{m}\right\}  ,\gamma,t}\left(
X\right)  =\\
&  =\int...\int q_{\gamma\left(  n-1\right)  ,\lambda_{\gamma\left(
n-1\right)  },t}\left(  x_{1}\right)  q_{\gamma\left(  n-2\right)
,\lambda_{\gamma\left(  n-2\right)  },t-x_{1}}\left(  x_{2}\right)  \times\\
&  \times...q_{\gamma\left(  0\right)  ,\lambda_{\gamma\left(  0\right)
},t-x_{1}-...-x_{n-1}}\left(  X-x_{1}-...-x_{n-1}\right)  dx_{1}%
...dx_{n_{n-1}};
\end{align*}
these kernels, however, do depend on all the functions $\left\{  \lambda
_{1},...,\lambda_{m}\right\}  $ at earlier times. To extract the relaxation
properties of the functions $\lambda_{i}$ from the representation $\left(
\ref{62}\right)  $ we again need to check that the kernels $Q$ have the
properties 1, 2 listed above. Now we do not have the property that the mean
number of clients at every node is conserved in the NLMP. However in our
closed network we have the fact that the mean number of clients in the whole
network is conserved, and that implies that at every node the mean number of
clients is bounded, which property is sufficient for our purposes, as the
analysis in \cite{RS1} shows.

\end{document}